\documentclass[10pt]{amsart}
\usepackage[leqno]{amsmath}
\usepackage{amssymb,latexsym,soul,cite,amsthm,color,enumitem,graphicx,mathtools,microtype,accents}
\usepackage[colorlinks=true,urlcolor=cerulean,citecolor=cerulean,linkcolor=cerulean,linktocpage,pdfpagelabels,bookmarksnumbered,bookmarksopen]{hyperref}
\definecolor{cerulean}{rgb}{0.0, 0.48, 0.65}
\usepackage[english]{babel}
\usepackage[left=2.5cm,right=2.5cm,top=2.5cm,bottom=2.5cm]{geometry}

\numberwithin{equation}{section}

\newtheorem{theorem}{Theorem}[section]
\theoremstyle{plain}
\newtheorem{lemma}[theorem]{Lemma}
\theoremstyle{plain}
\newtheorem{proposition}[theorem]{Proposition}
\theoremstyle{plain}

\theoremstyle{definition}
\newtheorem{remark}[theorem]{Remark}

\newcommand{\N}{{\mathbb N}}
\newcommand{\Z}{{\mathbb Z}}
\newcommand{\R}{{\mathbb R}}
\newcommand{\eps}{\varepsilon}
\newcommand{\beq}{\begin{equation}}
\newcommand{\eeq}{\end{equation}}
\renewcommand{\le}{\leqslant}
\renewcommand{\ge}{\geqslant}

\newcommand{\w}{W^{s,p}_0(\Omega)}
\newcommand{\fpl}{(-\Delta)_p^s\,}
\newcommand{\sm}{\mathcal{S}_{s,p}(m)}
\newcommand{\s}{\mathbb{S}^1}

\makeatletter
\newcommand{\leqnomode}{\tagsleft@true}
\newcommand{\reqnomode}{\tagsleft@false}
\makeatother

\newenvironment{enumroman}{\begin{enumerate}

}{\end{enumerate}}

\title[Fractional $p$-Laplacian weighted eigenvalues]{Monotonicity of eigenvalues of the fractional $p$-Laplacian\\ with singular weights}

\author[A.\ Iannizzotto]{Antonio Iannizzotto}

\address[A.\ Iannizzotto]{Dipartimento di Matematica e Informatica
\newline\indent
Universit\`a di Cagliari
\newline\indent
Via Ospedale 72, 09124 Cagliari, Italy}
\email{antonio.iannizzotto@unica.it}


\begin{document}

\begin{abstract}
We study a nonlinear, nonlocal eigenvalue problem driven by the fractional $p$-Laplacian with an indefinite, singular weight chosen in an optimal class. We prove the existence of an unbounded sequence of positive variational eigenvalues and alternative characterizations of the first and second eigenvalues. Then, by means of such characterizations, we prove strict decreasing monotonicity of such eigenvalues with respect to the weight function.
\end{abstract}

\maketitle

\begin{center}
Version of \today\
\end{center}

\section{Introduction}\label{sec1}

\noindent
The present note is devoted to the study of the following nonlinear, nonlocal, weighted eigenvalue problem:
\beq\label{ep}
\begin{cases}
\fpl u = \lambda m(x)|u|^{p-2}u & \text{in $\Omega$} \\
u = 0 & \text{in $\Omega^c$.}
\end{cases}
\eeq
The data of problem \eqref{ep} are as follows: $\Omega\subset\R^N$ ($N\ge 2$) is a bounded open domain, $p>1$, $s\in(0,1)$ are real numbers s.t.\ $ps<N$, and the leading operator is the fractional $p$-Laplacian, defined for any $u:\R^N\to\R$ smooth enough by
\[\fpl u(x) = 2\lim_{\eps\to 0^+}\int_{B_\eps^c(x)}\frac{|u(x)-u(y)|^{p-2}(u(x)-u(y))}{|x-y|^{N+ps}}\,dy.\]
Besides, $m$ is a possibly singular weight function of indefinite sign, chosen in the class
\[\mathcal{W}_{s,p} = \big\{m\in L^\frac{N}{ps}(\Omega):\,m>0 \ \text{on a subset of $\Omega$ with positive measure}\big\}.\]
Such choice of the weight function is a natural one. Indeed, the weak solutions of problem \eqref{ep} are elements of the fractional Sobolev space $W^{s,p}(\R^N)$, hence integrable up to the critical power $p^*_s=Np/(N-ps)$. So, for any $m\in L^\frac{N}{ps}(\Omega)$ and $u\in W^{s,p}(\R^N)$ the integral
\[\int_\Omega m(x)|u|^p\,dx\]
is finite. Clearly, our choice includes the most frequent case in the literature, namely $m\in L^\infty(\Omega)$. We say that $\lambda\in\R$ is an eigenvalue, if problem \eqref{ep} has a solution $u\neq 0$, which is then an eigenfunction associated to $\lambda$ (shortly, a $\lambda$-eigenfunction). 
\vskip2pt
\noindent
The study of eigenvalue problems for the fractional $p$-Laplacian (with constant weight $m=1$) started with \cite{LL}, where the authors were mainly concerned with the principal eigenvalue and its asymptotic behavior as $p\to\infty$. Then, several properties of the spectrum were investigated, such as simplicity of the first eigenvalue \cite{FP}, characterizations of the second one \cite{BP}, Weyl-type asymptotic laws \cite{IS}, calculations of the corresponding critical groups \cite{ILPS}.
\vskip2pt
\noindent
The purpose of this note is not only to extend some of the results above to the weighted case, but primarily to prove some useful monotonicity properties of the first and second eigenvalues with respect to the weight $m$. Our main result is the following:

\begin{theorem}\label{main}
Let $\Omega\subset\R^N$ ($N\ge 2$) be a bounded open domain, $p>1$, $s\in(0,1)$ s.t.\ $ps<N$. Then, for any $m\in\mathcal{W}_{s,p}$ problem \eqref{ep} admits an unbounded sequence of positive eigenvalues
\[0 < \lambda_1(m) < \lambda_2(m) \le \ldots \le \lambda_k(m) \le \ldots \to \infty.\]
In particular, $\lambda_1(m)$ is simple, isolated, and with constant sign associated eigenfunctions, while any eigenfunction associated to an eigenvalue $\lambda>\lambda_1(m)$ is nodal. Also, $\lambda_2(m)$ is the smallest eigenvalue above $\lambda_1(m)$. Finally, for any $m,\tilde m\in\mathcal{W}_{s,p}$ s.t.\ $m\le\tilde m$ in $\Omega$ the following hold:
\begin{enumroman}
\item\label{main1} $\lambda_k(m)\ge\lambda_k(\tilde m)$ for all $k\in\N$;
\item\label{main2} if $m\neq\tilde m$, then $\lambda_1(m) > \lambda_1(\tilde m)$;
\item\label{main3} if $m<\tilde m$ in $\Omega$, then $\lambda_2(m) > \lambda_2(\tilde m)$.
\end{enumroman}
\end{theorem}

\noindent
The sequence $(\lambda_k(m))$ is constructed by means of a standard min-max formula of Lusternik-Schnirelman type, based on the Fadell-Rabinowitz cohomological index, as in \cite{IS}. We note that such a sequence does not necessarily cover the whole spectrum of \eqref{ep}. The properties of $\lambda_1(m)$ and $\lambda_2(m)$ are proved essentially as in \cite{BP}. Also, we will characterize $\lambda_1(m)$ by means of a Rayleigh quotient, and $\lambda_2(m)$ by means of symmetric orbits on a weighted $L^p$-sphere (as was done in \cite{DR} for the local $p$-Laplacian and in \cite{BP} for the fractional $p$-Laplacian with $m=1$).
\vskip2pt
\noindent
Regarding the monotonicity properties \ref{main1}-\ref{main3}, we note that general non-strict monotonicity \ref{main1} easily follows from the general min-max formula. Instead, strict decreasing monotonicity is a more delicate matter. The analogues of \ref{main2}, \ref{main3} for the $p$-Laplacian were originally proved in \cite{A,AT}, respectively. The approach of \cite{AT} for the strict decreasing monotonicity of $\lambda_2(m)$ with respect to $m$ does not work here, due to the nonlocal behavior of the operator: roughly speaking, the reason is that if $u$ is a $\lambda_2(m)$-eigenfunction (hence nodal), then $u^\pm$ are {\em not} eigenfunctions on the mutual nodal domains. We follow a different method, based on the above-mentioned alternative characterization of $\lambda_2(m)$.
\vskip2pt
\noindent
In the linear case $p=2$, the sequence $(\lambda_k(m))$ can be equivalently defined by standard Courant-Fischer formulas and it covers the whole spectrum of the fractional Laplacian, see \cite{IP}. Strict decreasing monotonicity of $\lambda_k(m)$ with respect to $m$ (for any $k\in\N$) is then equivalent to some type of unique continuation property for the corresponding eigenfunctions, see \cite{FI}. 
\vskip2pt
\noindent
Theorem \ref{main} can be very useful in the study of general nonlinear equations of the type
\beq\label{np}
\begin{cases}
\fpl u = f(x,u) & \text{in $\Omega$} \\
u = 0 & \text{in $\Omega^c$,}
\end{cases}
\eeq
where $f:\Omega\times\R\to\R$ is a Carath\'eodory map with asymptotically $(p-1)$-linear behavior at infinity and/or at the origin, namely, s.t.\ the quotient
\[\frac{f(x,t)}{|t|^{p-2}t}\]
is bounded for $|t|$ big and/or small enough. Indeed, the existence of nontrivial weak solutions of problem \eqref{np} can be proved via topological methods (Browder's degree, index formula), by comparing the asymptotic behavior of the quotient above to a convenient weight function $m$ and exploiting the inequalities of Theorem \ref{main} \ref{main2}, \ref{main3}. An example of such application can be found in \cite{FI2}. Problems of the type \eqref{np} are also studied in \cite{FI1,ILPS,IL}.

\begin{remark}\label{ho}
In an independent line of research (see \cite{HPSS,HS}), the following class of weight functions was proposed: $m\in\widetilde{\mathcal{W}}_p$ if $m$ is a measurable function, s.t.\ $m{\rm d}_\Omega^{sa}\in L^r(\Omega)$, where ${\rm d}_\Omega(x)={\rm dist}(x,\Omega^c)$ and $a\in[0,1]$, $r>1$ are s.t.\
\[\frac{1}{r}+\frac{a}{p}+\frac{p-a}{p^*_s} < 1.\]
While $L^r(\Omega)\subset\widetilde{\mathcal{W}}_p$ for all $r>N/(ps)$, an easy calculation shows that $\mathcal{W}_{s,p}\not\subseteq\widetilde{\mathcal{W}}_p$. On the other hand, \cite[Example 2.3]{HS} shows that in general $\widetilde{\mathcal{W}}_p\not\subseteq\mathcal{W}_{s,p}$. We believe that the monotonicity properties of $\lambda_1(m)$, $\lambda_2(m)$ can be proved even for $m\in\widetilde{\mathcal{W}}_p$, using \cite[Lemma 2.6]{HS} in the place of Lemma \ref{cmp} below.
\end{remark}

\begin{remark}\label{low}
The assumption $ps<N$ is not an essential one: if $ps>N$, we believe that our arguments still hold with $m\in L^1(\Omega)$, and in the threshold case $ps=N$ with $m\in L^q(\Omega)$, for any $q>1$.
\end{remark}

\noindent
The structure of the paper is the following: in Section \ref{sec2} we recall some preliminary results on fractional Sobolev spaces; in Section \ref{sec3} we construct the sequence of variational eigenvalues and prove the characterizations of $\lambda_1(m)$, $\lambda_2(m)$; finally, in Section \ref{sec4} we prove the monotonicity properties of Theorem \ref{main}.
\vskip4pt
\noindent
{\bf Notation:} Throughout the paper, for any $A\subset\R^N$ we shall set $A^c=\R^N\setminus A$. For any two measurable functions $f,g:\Omega\to\R$, $f\le g$ in $\Omega$ will mean that $f(x)\le g(x)$ for a.e.\ $x\in\Omega$ (and similar expressions). The positive and negative parts of $f$ are $f^\pm=\max\{0,\pm f\}$, respectively. For all $q\in[1,\infty]$, $\|\cdot\|_q$ denotes the standard norm of $L^q(\Omega)$ (or $L^q(\R^N)$, which will be clear from the context). Every function $u$ defined in $\Omega$ will be identified with its $0$-extension to $\R^N$. Moreover, $C$ will denote a positive constant (whose value may change case by case).

\section{Preliminaries}\label{sec2}

\noindent
We recall some notions about fractional Sobolev spaces, referring the reader to \cite{BLP,DPV,ILPS} for details. Let $\Omega$, $s$, $p$ be as in Section \ref{sec1}. First, for all measurable $u:\R^n\to\R$ define the Gagliardo seminorm
\[[u]_{s,p} = \Big[\iint_{\R^N\times\R^N}\frac{|u(x)-u(y)|^p}{|x-y|^{N+ps}}\,dx\,dy\Big]^\frac{1}{p}.\]
We define the fractional Sobolev space
\[W^{s,p}(\R^N) = \big\{u\in L^p(\R^N):\,[u]_{s,p}<\infty\big\}.\]
We also define the subspace
\[\w = \big\{u\in W^{s,p}(\R^N):\,u=0 \ \text{in $\Omega^c$}\big\},\]
which is a uniformly convex (hence, reflexive) Banach space under the norm $\|u\|=[u]_{s,p}$, with dual space denoted by $W^{-s,p'}(\Omega)$. Set
\[p^*_s = \frac{Np}{N-ps},\]
then the embedding $\w\hookrightarrow L^q(\Omega)$ is continuous for all $q\in[1,p^*_s]$ and compact for all $q\in[1,p^*_s)$ (these properties are proved as in \cite[Lemma 2.4, Theorem 2.7]{BLP}).
\vskip2pt
\noindent
Now fix $m\in\mathcal{W}_{s,p}$. For all $\lambda\in\R$, we say that $u\in\w$ is a (weak) solution of problem \eqref{ep}, if for all $v\in\w$
\beq\label{wep}
\iint_{\R^N\times\R^N}\frac{|u(x)-u(y)|^{p-2}(u(x)-u(y))(v(x)-v(y))}{|x-y|^{N+ps}}\,dx\,dy = \lambda\int_\Omega m(x)|u|^{p-2}uv\,dx.
\eeq
Accordingly, we say that $\lambda\in\R$ is an eigenvalue of \eqref{ep}, if there exists a solution $u\in\w\setminus\{0\}$ of \eqref{ep}, which is then called an eigenfunction associated to $\lambda$ (or $\lambda$-eigenfunction). As pointed out in Section \ref{sec1}, the assumption $m\in\mathcal{W}_{s,p}$ makes equation \eqref{wep} well-posed. Indeed, from H\"older's inequality with three exponents we have
\[\int_\Omega\big|m(x)|u|^{p-2}uv\big|\,dx \le \|m\|_\frac{N}{ps}\|u\|_{p^*_s}^{p-1}\|v\|_{p^*_s} < \infty.\]
Nevertheless, the presence of the weight demands for some specific compactness results:

\begin{lemma}\label{cmp}
Let $(u_n)$ be a bounded sequence in $\w$, $m\in\mathcal{W}_{s,p}$. Then, there exists $u\in\w$ s.t.\ up to a subsequence
\begin{enumroman}
\item\label{cmp1} $\displaystyle\lim_n\int_\Omega m(x)|u_n|^p\,dx = \int_\Omega m(x)|u|^p\,dx$;
\item\label{cmp2} $\displaystyle\lim_n\int_\Omega m(x)|u_n-u|^p\,dx = 0$.
\end{enumroman}
\end{lemma}
\begin{proof}
First we recall some elementary inequalities. By convexity, for all $a,b\ge 0$ we have
\beq\label{conv}
a^p+b^p \le 2^{p-1}(a+b)^p.
\eeq
Also, for all $q>0$ and all $a,b>0$ we have
\beq\label{lag}
|a^q-b^q| \le q\max\{a^{q-1},\,b^{q-1}\}|a-b|
\eeq
(if $q\ge 1$, then \eqref{lag} also holds if $a,b\ge 0$). By reflexivity of $\w$ and the compact embedding $\w\hookrightarrow L^p(\Omega)$, we can find $u\in\w$ and a subsequence, still denoted $(u_n)$, s.t.\ $u_n\rightharpoonup u$ in $\w$, $u_n\to u$ in $L^p(\Omega)$, and $u_n(x)\to u(x)$ for a.e.\ $x\in\Omega$.
\vskip2pt
\noindent
We now prove \ref{cmp1}. Since $C^\infty_c(\Omega)$ is a dense subset of $L^\frac{N}{ps}(\Omega)$, there exists a sequence $(m_k)$ in $C^\infty_c(\Omega)$ s.t.\ $m_k\to m$ in $L^\frac{N}{ps}(\Omega)$. Using \eqref{conv}, \eqref{lag} (with $q=p$), and H\"older's inequality, for all $n,k\in\N$ we have
\begin{align*}
\Big|\int_\Omega m(x)\big(|u_n|^p-|u|^p\big)\,dx\Big| &\le \Big|\int_\Omega(m(x)-m_k(x))\big(|u_n|^p-|u|^p\big)\,dx\Big|+\Big|\int_\Omega m_k(x)\big(|u_n|^p-|u|^p\big)\,dx\Big| \\
&\le \int_\Omega|m(x)-m_k(x)|\big||u_n|^p-|u|^p\big|\,dx+\int_\Omega|m_k(x)|\big||u_n|^p-|u|^p\big|\,dx \\
&\le \int_\Omega|m(x)-m_k(x)|\big(|u_n|^p+|u|^p\big)\,dx+p\int_\Omega|m_k(x)|\max\{|u_n|^{p-1},|u|^{p-1}\}|u_n-u|\,dx \\
&\le 2^{p-1}\int_\Omega|m(x)-m_k(x)|\big(|u_n|+|u|\big)^p\,dx+p\int_\Omega|m_k(x)|\big(|u_n|+|u|\big)^{p-1}|u_n-u|\,dx \\
&\le 2^{p-1}\|m-m_k\|_\frac{N}{ps}\big\|(|u_n|+|u|)\big\|_{p^*_s}^p+p\|m_k\|_\infty\big\|(|u_n|+|u|)\big\|_{p}^{p-1}\|u_n-u\| \\
&\le C\|m-m_k\|_\frac{N}{ps}+C\|m_k\|_\infty\|u_n-u\|_p,
\end{align*}
with $C>0$ independent of $n$, $k$. Passing to the limit as $n\to\infty$, we get for all $k\in\N$
\[\limsup_n\Big|\int_\Omega m(x)\big(|u_n|^p-|u|^p\big)\,dx\Big| \le C\|m-m_k\|_\frac{N}{ps}.\]
Finally, letting $k\to\infty$ we have
\[\lim_n\Big|\int_\Omega m(x)\big(|u_n|^p-|u|^p\big)\,dx\Big| = 0.\]
Next we prove \ref{cmp2}. Define the subdomains
\[\Omega^+ = \big\{x\in\Omega:\,m(x)>0\big\}, \ \Omega^- = \big\{x\in\Omega:\,m(x)<0\big\}.\]
Assume $p\ge 2$, then by applying Clarkson's first inequality to the space $L^p(\Omega^+,m(x)\,dx)$ we get for all $n\in\N$
\beq\label{cmp3}
\int_{\Omega^+}m(x)\Big|\frac{u_n+u}{2}\Big|^p\,dx+\int_{\Omega^+}m(x)\Big|\frac{u_n-u}{2}\Big|^p\,dx \le \frac{1}{2}\Big[\int_{\Omega^+}m(x)|u_n|^p\,dx+\int_{\Omega^+}m(x)|u|^p\,dx\Big].
\eeq
Besides, by convexity we have
\[\int_{\Omega^+}m(x)\Big|\frac{u_n+u}{2}\Big|^p\,dx \ge \int_{\Omega^+}m(x)|u|^p\,dx+\frac{p}{2}\int_{\Omega^+}m(x)|u|^{p-2}u(u_n-u)\,dx.\]
Since $m|u|^{p-2}u\in W^{-s,p'}(\Omega)$, by weak convergence we have
\[\int_{\Omega^+}m(x)|u|^{p-2}u(u_n-u)\,dx \to 0.\]
Using all these relations in \eqref{cmp3}, we have for all $n\in\N$
\begin{align*}
\frac{1}{2^p}\int_{\Omega^+}m(x)|u_n-u|^p\,dx &\le \frac{1}{2}\Big[\int_{\Omega^+}m(x)|u_n|^p\,dx+\int_{\Omega^+}m(x)|u|^p\,dx\Big]-\int_{\Omega^+}m(x)\Big|\frac{u_n+u}{2}\Big|^p\,dx \\
&\le \frac{1}{2}\int_{\Omega^+}m(x)|u_n|^p\,dx-\frac{1}{2}\int_{\Omega^+}m(x)|u|^p\,dx-\frac{p}{2}\int_{\Omega^+}m(x)|u|^{p-2}u(u_n-u)\,dx,
\end{align*}
and the latter goes to $0$ as $n\to\infty$, due to \ref{main1}. So
\[\int_{\Omega^+}m(x)|u_n-u|^p\,dx \to 0.\]
Similarly we have
\[\int_{\Omega^-}m(x)|u_n-u|^p\,dx \to 0.\]
The sum of the two limits above gives \ref{main2}. If $p\in(1,2)$, then we argue similarly using Clarkson's second inequality in the place of \eqref{cmp3}.
\end{proof}

\begin{remark}
An alternative functional-analytic framework can be given for problem \eqref{ep}, by replacing the space $\w$ with $\widetilde{W}^{s,p}_0(\Omega)$, defined as the completion of $C^\infty_c(\Omega)$ with respect to the norm $[u]_{s,p}$ (see for instance \cite{BP}). The main motivation of our choice of the space $\w$ is to easily connect the results of this paper with most recent works on nonlinear fractional equations (see for instance \cite{FI1,FI2,ILPS,IL}). Also, we note that we can work under no special assumptions on the boundary $\partial\Omega$, since we do not aim at precise regularity information on the eigenfunctions of \eqref{ep}.
\end{remark}

\section{Existence and characterization of eigenvalues}\label{sec3}

\noindent
In the present section we assume $m\in\mathcal{W}_{s,p}$ and we prove the existence and some properties of variational eigenvalues of \eqref{ep}. First, set for all $u\in\w$
\[\Phi(u) = \|u\|^p, \ \Psi_m(u) = \int_\Omega m(x)|u|^p\,dx.\]
It is easily seen that $\Phi\in C^1(\w)$ with derivative given by
\[\langle\Phi'(u),v\rangle = p\iint_{\R^N\times\R^N}\frac{|u(x)-u(y)|^{p-2}(u(x)-u(y))(v(x)-v(y))}{|x-y|^{N+ps}}\,dx\,dy.\]
The map $\Phi':\w\to W^{-s,p'}(\Omega)$ satisfies the $(S)_+$-property, i.e., whenever $(u_n)$ is a sequence in $\w$ s.t.\ $u_n\rightharpoonup u$ in $\w$ and
\[\limsup_n\langle\Phi'(u_n),u_n-u\rangle \le 0,\]
then $u_n\to u$ in $\w$ (see \cite[Lemma 2.1]{FI1}). Also, $\Psi_m\in C^1(\w)$ with
\[\langle\Psi'_m(u),v\rangle = p\int_\Omega m(x)|u|^{p-2}uv\,dx.\]
The only critical value of $\Psi_m$ is $0$. Indeed, if $u\in\w$ is a critical point of $\Psi_m$, then we have
\[0 = \langle\Psi'_m(u),u\rangle = p\Psi_m(u).\]
Clearly, for any $\lambda\in\R$, equation \eqref{wep} is equivalent to having in $W^{-s,p'}(\Omega)$
\[\Phi'(u)-\lambda\Psi'_m(u) = 0.\]
Since $1$ is a regular value for $\Psi_m$, we can define a $C^1$-manifold in $\w$ by setting
\[\sm = \big\{u\in\w:\,\Psi_m(u)=1\big\}.\]
We must now recall some topological notions. First we denote by $C_2(A,B)$ the set of odd, continuous maps from $A$ to $B$ (both denoting subsets of some topological vector spaces). For any closed subset $M$ of a topological vector space, s.t.\ $M=-M$, we denote by $i(M)\in\N\cup\{\infty\}$ the Fadell-Rabinowitz cohomological index, introduced in \cite{FR} (for general properties of $i(\cdot)$ we refer to \cite[Section 2.5]{PAO}), with $\Z_2$ as an acting group.
\vskip2pt
\noindent
For all $k\in\N$ set
\[\mathcal{F}_k = \big\{M\subseteq \sm:\,M=-M \ \text{closed, $i(M)\ge k$}\big\}.\]
Then we define a sequence of variational (Lusternik-Schnirelman) eigenvalues of \eqref{ep} by setting for all $k\in\N$
\beq\label{lk}
\lambda_k(m) = \inf_{M\in\mathcal{F}_k}\sup_{u\in M}\Phi(u).
\eeq
Indeed, we have the following result:

\begin{proposition}\label{exi}
For all $k\in\N$, $\lambda_k(m)$ defined by \eqref{lk} is an eigenvalue of \eqref{ep}. Also, $(\lambda_k(m))$ is a nondecreasing sequence and
\[\lim_k \lambda_k(m) = \infty.\]
\end{proposition}
\begin{proof}
Our argument follows that of \cite[Proposition 2.2]{IS}. Let $\tilde\Phi$ denote the restriction of $\Phi$ to $\sm$. First, we prove that $\tilde\Phi$ satisfies the Palais-Smale condition. Let $(u_n)$ be a sequence in $\sm$ with the following properties: $\Phi(u_n)\to c\in\R$, and there exists a sequence $(\mu_n)$ in $\R$ s.t.\ in $W^{-s,p'}(\Omega)$
\beq\label{exi1}
\Phi'(u_n)-\mu_n\Psi'_m(u_n) \to 0.
\eeq
First, we note that $(u_n)$ is bounded in $\w$, hence passing to a subsequence we have $u_n\rightharpoonup u$ in $\w$. By Lemma \ref{cmp} \ref{cmp1}, passing to a further subsequence we have
\[\Psi_m(u) = \lim_n\Psi_m(u_n) = 1,\]
hence $u\in\sm$. Testing \eqref{exi1} with $u_n$ we get
\[\Phi(u_n)-\mu_n\Psi_m(u_n) \to 0,\]
hence $\mu_n\to c$ (in particular, $(\mu_n)$ is bounded). Further, testing \eqref{exi1} with $u_n-u$, using the boundedness of $(\mu_n)$, and H\"older's inequality, we get
\begin{align*}
\langle\Phi'(u_n),u_n-u\rangle &= p\mu_n\int_\Omega m(x)|u_n|^{p-2}u_n(u_n-u)\,dx+{\bf o}(1) \\
&\le C\Big[\int_\Omega |m(x)||u_n|^p\,dx\Big]^\frac{1}{p'}\Big[\int_\Omega |m(x)||u_n-u|^p\,dx\Big]^\frac{1}{p}+{\bf o}(1),
\end{align*}
and the latter tends to $0$ as $n\to\infty$ by Lemma \ref{cmp} \ref{cmp1} \ref{cmp2} (with weight $|m|$). By the $(S)_+$-property of $\Phi'$, we conclude that $u_n\to u$ in $\w$.
\vskip2pt
\noindent
Now fix $k\in\N$ and let $\lambda_k(m)$ be defined by \eqref{lk}. We prove that $\lambda_k(m)$ is a critical value of $\tilde\Phi$, arguing by contradiction: assume that $\lambda_k(m)$ is a regular value of $\tilde\Phi$. By the deformation theorem on $C^1$-manifolds (see \cite[Theorem 2.5]{B}), there exist $\eps>0$ and an odd homeomorphism $\eta:\sm\to\sm$ s.t.\ for all $u\in\sm$ the following implication holds:
\[\Phi(u)\le\lambda_k(m)+\eps \ \Longrightarrow \ \Phi(\eta(u))\le\lambda_k(m)-\eps.\]
By \eqref{lk}, there exists $M\in\mathcal{F}_k$ s.t.\
\[\sup_{u\in M}\Phi(u) \le \lambda_k(m)+\eps.\]
By the monotonicity property of the index (see \cite[Proposition 2.12 $(i_2)$]{PAO}), we have $i(\eta(M))\ge k$, hence $\eta(M)\in\mathcal{F}_k$. From the relations above we deduce
\[\sup_{u\in\eta(M)}\Phi(u) \le \lambda_k(m)-\eps,\]
against \eqref{lk}. So there exist $u\in\sm$, $\mu\in\R$ s.t.\ $\Phi(u)=\lambda_k(m)$ and in $W^{-s,p'}(\Omega)$
\[\Phi'(u)-\mu\Psi'_m(u) = 0.\]
Testing with $u$ we get
\[\lambda_k(m) = \|u\|^p = \mu\int_\Omega m(x)|u|^p\,dx = \mu.\]
Thus, as noted before, $u\in\sm$ satisfies \eqref{wep} with $\lambda=\lambda_k(m)>0$, i.e., $\lambda_k(m)$ is a positive eigenvalue with (normalized) associated eigenfunction $u$.
\vskip2pt
\noindent
To conclude the proof, we note that for all $k\in\N$ we have $\mathcal{F}_k\supseteq\mathcal{F}_{k+1}$, hence from \eqref{lk} we have $\lambda_k(m)\le\lambda_{k+1}(m)$. Finally, since ${\rm dim}(\w)=\infty$, by the neighborhood property of $i(\cdot)$ (see \cite[Proposition 2.12 $(i_8)$]{PAO}) we have $i(\sm)=\infty$. Besides,
\[\sup_{u\in\sm}\Phi(u) = \infty,\]
so we have $\lambda_k(m)\to\infty$ as $k\to\infty$.
\end{proof}

\begin{remark}
An analogous construction, performed replacing the index $i(\cdot)$ with Krasnoselskii's genus (see \cite[Definition 5.4.25]{GP} and the subsequent properties), leads to a possibly different sequence of eigenvalues. But, notably, the first two eigenvalues $\lambda_1(m)$, $\lambda_2(m)$ remain the same, due to the following results.
\end{remark}

\noindent
Now we prove further properties and alternative variational characterizations of the first two eigenvalues $\lambda_1(m)$, $\lambda_2(m)$, which will be used in Section \ref{sec4} to show monotonicity with respect to $m$. The following results are 'weighted' versions of those of \cite[Section 4]{BP}. We begin with the first eigenvalue:

\begin{proposition}\label{first}
The smallest positive eigenvalue of \eqref{ep} is simple, given by
\[\lambda_1(m) = \min_{u\in\sm}\|u\|^p > 0.\]
Also, for any eigenvalue $\lambda>0$ and any $\lambda$-eigenfunction $u\in\w$, the following hold:
\begin{enumroman}
\item\label{first1} if $\lambda=\lambda_1(m)$, then either $u>0$ in $\Omega$, or $u<0$ in $\Omega$;
\item\label{first2} if $\lambda>\lambda_1(m)$, then $u$ is nodal.
\end{enumroman}
Fro now on, we denote by $e_1\in\sm$ the unique positive, normalized $\lambda_1(m)$-eigenfunction.
\end{proposition}
\begin{proof}
Let $\lambda_1(m)$ be defined by \eqref{lk} with $k=1$. By Proposition \ref{exi}, $\lambda_1(m)$ is an eigenvalue, with an associated eigenfunction $u\in\sm$ s.t.\
\[\lambda_1(m) = \Phi(u) > 0.\]
For any $v\in\sm$ set $M=\{v,\,-v\}$, then by the definiteness property of $i(\cdot)$ (see \cite[Proposition 2.12 $(i_1)$]{PAO}) we have $M\in\mathcal{F}_1$ with
\[\Phi(v) = \Phi(-v) \ge \lambda_1(m).\]
Taking the infimum over $v\in\sm$, we prove the characterization of $\lambda_1(m)$. Now let $\lambda>0$ be an eigenvalue, with associated eigenfunction $v\in\sm$, then
\[\lambda = \|v\|^p \ge \lambda_1(m).\]
So, $\lambda_1(m)$ is the smallest positive eigenvalue.
\vskip2pt
\noindent
Now we prove \ref{first1}. Without loss of generality, let $u\in\sm$ be a $\lambda_1(m)$-eigenfunction s.t.\ $u^+\neq 0$. Then $|u|\in\sm$ and
\begin{align*}
\Phi(|u|) &= \iint_{\R^N\times\R^N}\frac{||u(x)|-|u(y)||^p}{|x-y|^{N+ps}}\,dx\,dy \\
&\le \iint_{\R^N\times\R^N}\frac{|u(x)-u(y)|^p}{|x-y|^{N+ps}}\,dx\,dy = \lambda_1(m),
\end{align*}
with strict inequality if $u$ changes sign in $\Omega$. By the characterization of $\lambda_1(m)$, we deduce that $u\ge 0$ in $\Omega$. Further, by the strong maximum principle \cite[Proposition B.3]{BSY} (with $a=-\lambda_1(m)m$, $\mu=0$) we have $u>0$ in $\Omega$. Similarly, if $u^-\neq 0$ we deduce $u<0$ in $\Omega$.
\vskip2pt
\noindent
Next we prove that $\lambda_1(m)$ is simple. Without loss of generality, let $u,v\in\sm$ be $\lambda_1(m)$-eigenfunctions s.t.\ $u,v>0$ in $\Omega$. Our argument is based on 'geodesic convexity': set for all $t\in[0,1]$
\[w_t = \big[(1-t)u^p+tv^p\big]^\frac{1}{p}.\]
We have $w_t\in\w$. Indeed, by \cite[Lemma 4.1]{FP}
\beq\label{first3}
\|w_t\|^p \le (1-t)\|u\|^p+t\|v\|^p = \lambda_1(m).
\eeq
Besides,
\[\int_\Omega m(x)w_t^p\,dx = (1-t)\int_\Omega m(x)u^p\,dx+t\int_\Omega m(x)v^p\,dx = 1,\]
hence $w_t\in\sm$. By the characterization we have $\|w_t\|^p\ge \lambda_1(m)$, which along with \eqref{first3} gives $\|w_t\|^p=\lambda_1(m)$. So, \eqref{first3} rephrases as an equality and we have
\begin{align}\label{first4}
&\iint_{\R^N\times\R^N}\Big|\big[(1-t)u^p(x)+tv^p(x)\big]^\frac{1}{p}-\big[(1-t)u^p(y)+tv^p(y)\big]^\frac{1}{p}\Big|^p\frac{dx\,dy}{|x-y|^{N+ps}} \\
\nonumber &= \iint_{\R^N\times\R^N}\big[(1-t)|u(x)-u(y)|^p+t|v(x)-v(y)|^p\big]\frac{dx\,dy}{|x-y|^{N+ps}} \ (= \lambda_1(m)).
\end{align}
We focus on the integrands of \eqref{first4}, setting for all $x,y\in\R^N$
\[\xi = \big((1-t)^\frac{1}{p}u(x),\,t^\frac{1}{p}v(x)\big), \ \eta = \big((1-t)^\frac{1}{p}u(y),\,t^\frac{1}{p}v(y)\big).\]
By the triangular inequality in $\ell^p$, we have
\begin{align*}
0 &\le \Big|\big[(1-t)u^p(x)+tv^p(x)\big]^\frac{1}{p}-\big[(1-t)u^p(y)+tv^p(y)\big]^\frac{1}{p}\Big|^p \\
&= \big|\|\xi\|_{\ell^p}-\|\eta\|_{\ell^p}\big|^p \\
&\le \|\xi-\eta\|_{\ell^p}^p \\
&= \big[(1-t)|u(x)-u(y)|^p+t|v(x)-v(y)|^p\big].
\end{align*}
So, the integrands in \eqref{first4} coincide in $\R^N\times\R^N$. This, in turn, implies that $\xi=\alpha\eta$ for some $\alpha\in\R$, i.e., for a.e.\ $x,y\in\R^N$ we have
\[\frac{u(x)}{u(y)} = \frac{v(x)}{v(y)} \ (=\alpha).\]
Recalling that $u,v\in\sm$ we see that $\alpha=1$, hence $u=v$.
\vskip2pt
\noindent
Finally we prove \ref{first2}. Let $\lambda>\lambda_1(m)$ be an eigenvalue and assume, by contradiction, that there exists a $\lambda$-eigenfunction $u\in\sm$ s.t.\ $u\ge 0$ in $\Omega$. Again by \cite[Proposition B.3]{BSY}, we have in fact $u>0$ in $\Omega$. Besides, let $e_1\in\sm$ be defined as above. For any $\eps>0$ set
\[v_\eps = \min\Big\{e_1,\,\frac{1}{\eps}\Big\}.\]
We claim that $v_\eps\in\w$ with $\|v_\eps\|^p\le\lambda_1(m)$. Indeed, clearly $v_\eps\in L^\infty(\R^N)$ and $v_\eps=0$ in $\Omega^c$, so $v_\eps\in L^p(\R^N)$. Also, since the map
\[t\mapsto\min\Big\{t,\,\frac{1}{\eps}\Big\}\]
is $1$-Lipschitz continuous, we have
\[\iint_{\R^N\times\R^N}\frac{|v_\eps(x)-v_\eps(y)|^p}{|x-y|^{N+ps}}\,dx\,dy \le \iint_{\R^N\times\R^N}\frac{|e_1(x)-e_1(y)|^p}{|x-y|^{N+ps}}\,dx\,dy = \lambda_1(m).\]
Set further
\[w_\eps = \frac{v_\eps^p}{(u+\eps)^{p-1}}.\]
We have $w_\eps\in\w$. Indeed, we have $0\le w_\eps\le \eps^{1-2p}$ in $\Omega$ and $w_\eps=0$ in $\Omega^c$, so $w_\eps\in L^p(\R^N)$. In order to estimate the $\w$-norm of $w_\eps$, we use twice \eqref{lag} (with $q=p,p-1$, respectively) and get for a.e.\ $x,y\in\R^N$
\begin{align*}
|w_\eps(x)-w_\eps(y)| &= \Big|\frac{v_\eps^p(x)}{(u(x)+\eps)^{p-1}}-\frac{v_\eps^p(y)}{(u(y)+\eps)^{p-1}}\Big| \\
&\le \Big|\frac{v_\eps^p(x)-v_\eps^p(y)}{(u(x)+\eps)^{p-1}}\Big|+v_\eps^p(y)\Big|\frac{(u(x)+\eps)^{p-1}-(u(y)+\eps)^{p-1}}{(u(x)+\eps)^{p-1}(u(y)+\eps)^{p-1}}\Big| \\
&\le \frac{p}{\eps^{p-1}}\max\big\{v_\eps^{p-1}(x),\,v_\eps^{p-1}(y)\big\}|v_\eps(x)-v_\eps(y)| \\
&+ \frac{p-1}{\eps^p}\frac{\max\big\{(u(x)+\eps)^{p-2},\,(u(y)+\eps)^{p-2}\big\}}{(u(x)+\eps)^{p-1}(u(y)+\eps)^{p-1}}|u(x)-u(y)| \\
&\le \frac{p}{\eps^{2p-2}}|v_\eps(x)-v_\eps(y)|+\frac{p-1}{\eps^{p+1}}\max\Big\{\frac{1}{(u(x)+\eps)^{p-1}},\,\frac{1}{(u(y)+\eps)^{p-1}}\Big\}|u(x)-u(y)| \\
&\le \frac{p}{\eps^{2p-2}}|e_1(x)-e_1(y)|+\frac{p-1}{\eps^{2p}}|u(x)-u(y)|.
\end{align*}
Integrating over $\R^N\times\R^N$, we find $C_\eps>0$ s.t.\
\[\iint_{\R^N\times\R^N}\frac{|w_\eps(x)-w_\eps(y)|^p}{|x-y|^{N+ps}}\,dx\,dy \le C_\eps\big(\|e_1\|^p+\|u\|^p\big).\]
By Picone's discrete inequality \cite[Proposition 4.2]{BF} we have for a.e.\ $x,y\in\R^N$:
\beq\label{first5}
|v_\eps(x)-v_\eps(y)|^p \ge |u(x)-u(y)|^{p-2}(u(x)-u(y))(w_\eps(x)-w_\eps(y)).
\eeq
Testing \eqref{wep} with $w_\eps\in\w$, using \eqref{first5}, and recalling that $\|v_\eps\|^p\le\lambda_1(m)$, we get
\begin{align*}
\lambda_1(m) &\ge \iint_{\R^N\times\R^N}\frac{|v_\eps(x)-v_\eps(y)|^p}{|x-y|^{N+ps}}\,dx\,dy \\
&\ge \iint_{\R^N\times\R^N}\frac{|u(x)-u(y)|^{p-2}(u(x)-u(y))(w_\eps(x)-w_\eps(y))}{|x-y|^{N+ps}}\,dx\,dy \\
&= \lambda\int_\Omega m(x)|u|^{p-2}uw_\eps\,dx \\
&= \lambda\int_\Omega m(x)\Big[\frac{u}{u+\eps}\Big]^{p-1}v_\eps^p\,dx.
\end{align*}
Now let $\eps\to 0^+$. Then we have in $\Omega$
\[v_\eps^p \to e_1^p, \ \Big[\frac{u}{u+\eps}\Big]^{p-1}\to 1.\]
By Fatou's lemma, passing to the limit we get
\[\lambda_1(m) \ge \lambda\int_\Omega m(x)e_1^p\,dx = \lambda,\]
a contradiction. Similarly we argue if $u\le 0$. Thus, $u$ must change sign in $\Omega$.
\end{proof}

\noindent
The following result provides an alternative variational characterization of the second eigenvalue $\lambda_2(m)$, analogous to that introduced in \cite{DR} for the local case ($s=1$) and in \cite{BP} for the nonlocal case without weight ($m=1$):

\begin{proposition}\label{scd}
The smallest eigenvalue of \eqref{ep} above $\lambda_1(m)$ is
\[\lambda_2(m) = \inf_{f\in C_2(\s,\sm)}\max_{\omega\in\s}\|f(\omega)\|^p.\]
\end{proposition}
\begin{proof}
For now, let $\lambda_2(m)$ be defined by \eqref{lk} with $k=2$. We begin by proving that
\beq\label{scd1}
\lambda_1(m) < \lambda_2(m).
\eeq
From Proposition \ref{exi} we know that $\lambda_1(m)\le\lambda_2(m)$. Arguing by contradiction, assume that $\lambda_1(m)=\lambda_2(m)$. By \eqref{lk}, for any $n\in\N$ we can find $M_n\in\mathcal{F}_2$ s.t.\
\[\sup_{u\in M_n}\|u\|^p \le \lambda_1(m)+\frac{1}{n}.\]
For all $\rho>0$ define two relatively open subsets of $\sm$ by setting
\[B^\pm_\rho = \Big\{u\in\sm:\,\Big|\int_\Omega m(x)|u-(\pm e_1)|^p\,dx\Big|<\rho\Big\},\]
where $e_1\in\sm$ is defined as in Proposition \ref{first}. We distinguish two cases:
\begin{itemize}[leftmargin=1cm]
\item[$(a)$] If there exists a sequence $(\rho_n)$ s.t.\ $\rho_n\to 0^+$ and for all $n\in\N$
\[\Big\{u\in\sm:\,\|u\|^p\le\lambda_1(m)+\frac{1}{n}\Big\} \subseteq B^+_{\rho_n}\cup B^-_{\rho_n},\]
then in particular $M_n\subseteq B^+_{\rho_n}\cup B^-_{\rho_n}$. Since $M_n$ is symmetric, we have $M_n\cap B^\pm_{\rho_n}\neq\emptyset$, while for all $n\in\N$ big enough we have
\[M_n\cap B^+_{\rho_n}\cap B^-_{\rho_n}=\emptyset.\]
Otherwise, there would exist a sequence $(u_n)$ in $\sm$ s.t.\ for all $n\in\N$
\[\|u_n\|^p \le \lambda_1(m)+\frac{1}{n}\]
and
\beq\label{scd2}
\Big|\int_\Omega m(x)|u_n-(\pm e_1)|^p\,dx\Big| < \rho_n.
\eeq
From the first relation we see that $(u_n)$ is bounded in $\w$, hence passing to a subsequence we have $u_n\rightharpoonup u$ in $\w$, which in turn implies $\|u\|^p\le\lambda_1(m)$. By Lemma \ref{cmp} \ref{cmp1}, along a further subsequence we have
\[\int_\Omega m(x)|u|^p\,dx = \lim_n\int_\Omega m(x)|u_n|^p\,dx = 1,\]
hence $u\in\sm$. By Proposition \ref{first} (simplicity of $\lambda_1(m)$), the last relation implies either $u=e_1$ or $u=-e_1$. Assume $u=e_1$. Using again Lemma \ref{cmp} and passing to the limit in \eqref{scd2}, up to a subsequence we have
\[2^p\int_\Omega m(x)e_1^p\,dx = \lim_n\int_\Omega m(x)|u_n+e_1|^p\,dx = 0,\]
a contradiction. The case $u=-e_1$ is dealt with similarly. So we have split $M_n$ into two separate parts. Set for all $u\in M_n$
\[\varphi(u) = \begin{cases}
1 & \text{if $u\in M_n\cap B^+_{\rho_n}$} \\
-1 & \text{if $u\in M_n\cap B^-_{\rho_n}$.}
\end{cases}\]
Then $\varphi\in C_2(M_n,\R)$ and $\varphi(M_n)=\mathbb{S}^0$. By \cite[Example 2.11, Proposition 2.12 $(i_2)$]{PAO} we have
\[i(M_n) \le i(\mathbb{S}^0) = 1,\]
against $M_n\in\mathcal{F}_2$.
\item[$(b)$] Otherwise, there exist a sequence $(u_n)$ in $\sm$ and $\rho_0>0$, s.t.\ for all $n\in\N$ we have
\[\|u_n\|^p \le \lambda_1(m)+\frac{1}{n},\]
and $u_n\notin B^+_{\rho_0}\cup B^-_{\rho_0}$, which amounts to
\[\Big|\int_\Omega m(x)|u_n-(\pm e_1)|^p\,dx\Big| \ge \rho_0.\]
Reasoning as in case $(a)$, up to a subsequence we have either $u_n\rightharpoonup e_1$ or $u_n\rightharpoonup -e_1$ in $\w$. In either case, passing to the limit in the inequality above and using Lemma \ref{cmp} again, we reach a contradiction.
\end{itemize}
In both cases \eqref{scd1} is achieved. Now set
\[\tilde\lambda_2(m) = \inf_{f\in C_2(\s,\sm)}\max_{\omega\in\s}\|f(\omega)\|^p.\]
It is easily seen that
\beq\label{scd3}
\lambda_2(m) \le \tilde\lambda_2(m).
\eeq
Indeed, fix $f\in C_2(\s,\sm)$ and set $M=f(\s)$. Again by \cite[Example 2.11, Proposition 2.12 $(i_2)$]{PAO} we have
\[i(M) \ge i(\s) = 2,\]
hence $M\in\mathcal{F}_2$. Now \eqref{lk} implies
\[\lambda_2(m) \le \max_{u\in M}\|u\|^p = \max_{\omega\in\s}\|f(\omega)\|^p.\]
Taking the infimum over $f\in C_2(\s,\sm)$ we have \eqref{scd3}. Next, we prove that
\beq\label{scd4}
\tilde\lambda_2(m) = \min\big\{\lambda>\lambda_1(m): \ \text{$\lambda$ is an eigenvalue of \eqref{ep}}\big\}.
\eeq
First, $\tilde\lambda_2(m)$ is an eigenvalue by \cite[Proposition 2.7]{C}, and by concatenating \eqref{scd1}, \eqref{scd3} we have $\lambda_1(m)<\tilde\lambda_2(m)$. Further, let $\lambda>\lambda_1(m)$ be an eigenvalue, $u\in\w\setminus\{0\}$ be a $\lambda$-eigenfunction. By Proposition \ref{first} \ref{first2}, we have $u^\pm\neq 0$. Test \eqref{wep} with $u^+\in\w$ and use \cite[Lemma 2.1]{IL} to get
\begin{align*}
0 &< \|u^+\|^p \\
&\le \iint_{\R^N\times\R^N}\frac{|u(x)-u(y)|^{p-2}(u(x)-u(y))(u^+(x)-u^+(y))}{|x-y|^{N+ps}}\,dx\,dy \\
&= \lambda\int_\Omega m(x)|u|^{p-2}uu^+\,dx \\
&= \lambda\int_\Omega m(x)(u^+)^p\,dx,
\end{align*}
hence
\[\int_\Omega m(x)(u^+)^p\,dx > 0.\]
Similarly we get
\[\int_\Omega m(x)(u^-)^p\,dx > 0.\]
For all $x,y\in\R^N$ set
\[U = u^+(x)-u^+(y), \ V = u^-(x)-u^-(y)\]
(we omit the variables for notational simplicity). For any $\omega=(\omega_1,\omega_2)\in\s$, by \cite[Eq.\ (4.7)]{BP} we have
\beq\label{scd5}
|\omega_1U-\omega_1V|^{p-2}(\omega_1U-\omega_1V)\omega_1U-|\omega_2U-\omega_2V|^{p-2}(\omega_2U-\omega_2V)\omega_2V \ge |\omega_1U-\omega_2V|^p.
\eeq
Now test again \eqref{wep} with $u^\pm\in\w$ and recall that $u=u^+-u^-$ to get
\[\lambda\int_\Omega m(x)(u^+)^p\,dx = \iint_{\R^N\times\R^N}\frac{|U-V|^{p-2}(U-V)U}{|x-y|^{N+ps}}\,dx\,dy,\]
\[-\lambda\int_\Omega m(x)(u^-)^p\,dx = \iint_{\R^N\times\R^N}\frac{|U-V|^{p-2}(U-V)V}{|x-y|^{N+ps}}\,dx\,dy.\]
Multiply by $|\omega_1|^p$, $|\omega_2|^p$, respectively, the relations above, then subtract:
\begin{align}\label{scd6}
\lambda K^p &= \iint_{\R^N\times\R^N}\frac{|\omega_1U-\omega_1V|^{p-2}(\omega_1U-\omega_1V)\omega_1U-|\omega_2U-\omega_2V|^{p-2}(\omega_2U-\omega_2V)\omega_2V}{|x-y|^{N+ps}}\,dx\,dy \\
\nonumber &\ge \iint_{\R^N\times\R^N}\frac{|\omega_1U-\omega_2V|^p}{|x-y|^{N+ps}}\,dx\,dy,
\end{align}
where we have also used \eqref{scd5} and set
\[K = \Big[|\omega_1|^p\int_\Omega m(x)(u^+)^p\,dx+|\omega_2|^p\int_\Omega m(x)(u^-)^p\,dx\Big]^\frac{1}{p} > 0.\]
Further, set for any $\omega=(\omega_1,\omega_2)\in\s$
\[f(\omega) = \frac{\omega_1u^+-\omega_2u^-}{K}\in\w.\]
Clearly, $f:\s\to\w$ is a continuous, odd map. Plus, recalling that $u^+u^-=0$ in $\Omega$, we have
\[\int_\Omega m(x)|f(\omega)|^p\,dx = \frac{|\omega_1|^p}{K^p}\int_\Omega m(x)(u^+)^p\,dx+\frac{|\omega_2|^p}{K^p}\int_\Omega m(x)(u^-)^p\,dx = 1,\]
so $f\in C_2(\s,\sm)$. By the definition of $U$, $V$ and \eqref{scd6} we have for all $\omega\in\s$
\[\|f(\omega)\|^p = \frac{1}{K^p}\iint_{\R^N\times\R^N}\frac{|\omega_1U-\omega_2V|^p}{|x-y|^{N+ps}}\,dx\,dy \le \lambda.\]
Recalling the definition of $\tilde\lambda_2(m)$, we deduce
\[\tilde\lambda_2(m) \le \lambda,\]
thus proving \eqref{scd4}. To conclude, we note that from \eqref{scd1} and \eqref{scd4} it follows
\[\tilde\lambda_2(m) \le \lambda_2(m),\]
which along with \eqref{scd3} implies $\lambda_2(m)=\tilde\lambda_2(m)$.
\end{proof}

\begin{remark}
Unlike in the case $m\in L^\infty(\Omega)$, we cannot easily see that eigenfunctions are bounded in $\Omega$. In fact, this is not true in general for eigenvalue problems with singular weights, even in the local case $s=1$ with weight $m\in L^\frac{N}{p}(\Omega)$, see \cite[Example 2.2]{LP}.
\end{remark}

\section{Monotonicity of eigenvalues}\label{sec4}

\noindent
In this final section, we study the monotonicity of the maps $m\mapsto\lambda_k(m)$ ($k\in\N$) defined in \eqref{lk}, associating to any weight $m\in\mathcal{W}_{s,p}$ the variational eigenvalues of \eqref{ep}.
\vskip2pt
\noindent
For a general $k\in\N$ we prove non-strict nonincreasing monotonicity:

\begin{proposition}\label{mok}
Let $m,\tilde m\in\mathcal{W}_{s,p}$ be s.t.\ $m\le\tilde m$ in $\Omega$. Then, $\lambda_1(m)\ge\lambda_1(\tilde m)$ for all $k\in\N$.
\end{proposition}
\begin{proof}
Fix $\eps>0$, $k\in\N$. By \eqref{lk}, there exists $M\in\mathcal{F}_k$ s.t.\
\beq\label{mok1}
\sup_{u\in M}\|u\|^p \le \lambda_k(m)+\eps.
\eeq
Since $m\le\tilde m$ in $\Omega$, for all $u\in M$ we have
\[\int_\Omega\tilde m(x)|u|^p\,dx \ge \int_\Omega m(x)|u|^p\,dx = 1.\]
So we may set
\[\eta(u) = \frac{u}{\Big[\int_\Omega\tilde m(x)|u|^p\,dx\Big]^\frac{1}{p}} \in \mathcal{S}_{s,p}(\tilde m).\]
It is easily seen that $\eta$ is an odd map, with inverse given for all $v\in\eta(M)$ by
\[\eta^{-1}(v) = \frac{v}{\Big[\int_\Omega m(x)|v|^p\,dx\Big]^\frac{1}{p}}.\]
Both $\eta$, $\eta^{-1}$ are continuous by Lemma \ref{cmp} \ref{cmp1}, so $\eta:M\to\eta(M)$ is an odd homeomorphism. So the set $\eta(M)\subset\mathcal{S}_{s,p}(\tilde m)$ is closed and symmetric, and by \cite[Proposition 2.12 $(i_2)$]{PAO} we have $i(\eta(M))=i(M)\ge k$. Thus, by \eqref{lk} (with weight $\tilde m$) and \eqref{mok1} we have
\[\lambda_k(\tilde m) \le \sup_{v\in\eta(M)}\|v\|^p = \sup_{u\in M}\frac{\|u\|^p}{\int_\Omega\tilde m(x)|u|^p\,dx} \le \lambda_k(m)+\eps.\]
Letting $\eps\to 0^+$, we get $\lambda_k(\tilde m)\le\lambda_k(m)$.
\end{proof}

\noindent
When it comes to {\em strict} decreasing monotonicity, the above argument fails in general, due to a lack of compactness of the sets $M\in\mathcal{F}_k$. Nevertheless, by means of the alternative characterizations of $\lambda_1(m)$, $\lambda_2(m)$ proved in Section \ref{sec3} we are able to prove such property at least for $k=1,2$.
\vskip2pt
\noindent
For the first eigenvalue, we can prove strict decreasing monotonicity under a mild assumption:

\begin{proposition}\label{mof}
Let $m,\tilde m\in\mathcal{W}_{s,p}$ be s.t.\ $m\le\tilde m$ in $\Omega$, $m\neq\tilde m$. Then, $\lambda_1(m)>\lambda_1(\tilde m)$.
\end{proposition}
\begin{proof}
Let $e_1\in\sm$ be defined as in Proposition \ref{first}. Since $\tilde m>m$ on a positively measured subset of $\Omega$, and $e_1>0$ in $\Omega$, we have
\[\int_\Omega \tilde m(x)e_1^p\,dx > \int_\Omega m(x)e_1^p\,dx = 1.\]
Besides, $\|e_1\|^p=\lambda_1(m)$. Set now
\[v = \frac{e_1}{\Big[\int_\Omega \tilde m(x)e_1^p\,dx\Big]^\frac{1}{p}} \in \mathcal{S}_{s,p}(\tilde m).\]
By Proposition \ref{first} (variational characterization of $\lambda_1(\tilde m)$) we have
\[\lambda_1(\tilde m) \le \|v\|^p = \frac{\|e_1\|^p}{\int_\Omega \tilde m(x)e_1^p\,dx} < \lambda_1(m),\]
which concludes the proof.
\end{proof}

\noindent
Regarding the second eigenvalue, we prove strict decreasing monotonicity under a stronger assumption (as in \cite{AT}):

\begin{proposition}\label{mos}
Let $m,\tilde m\in\mathcal{W}_{s,p}$ be s.t.\ $m<\tilde m$ in $\Omega$. The, $\lambda_2(m)>\lambda_2(\tilde m)$.
\end{proposition}
\begin{proof}
For any $\lambda>\lambda_1(m)$ we set
\[C(m,\tilde m,\lambda) = \inf\Big\{\int_\Omega\tilde m(x)|u|^p\,dx:\,u\in\sm,\,\|u\|^p\le\lambda\Big\}\]
(the definition is well posed due to the characterization of $\lambda_1(m)$ from Proposition \ref{first}). We claim that
\beq\label{mos1}
C(m,\tilde m,\lambda) > 1.
\eeq
Indeed, since $m<\tilde m$ in $\Omega$, we clearly have $C(m,\tilde m,\lambda)\ge 1$. Further, let $(u_n)$ be a minimizing sequence in $\sm$ for $C(m,\tilde m,\lambda)$. Then, $(u_n)$ is bounded in $\w$. By reflexivity, passing to a subsequence we have $u_n\rightharpoonup u$ in $\w$, hence
\[\|u\|^p \le \liminf_n\|u_n\|^p \le \lambda.\]
By Lemma \ref{cmp} \ref{cmp1}, up to a further subsequence we have
\[\int_\Omega m(x)|u|^p\,dx = \lim_n\int_\Omega m(x)|u_n|^p\,dx = 1,\]
\[\int_\Omega\tilde m(x)|u|^p\,dx = \lim_n\int_\Omega\tilde m(x)|u_n|^p\,dx = C(m,\tilde m,\lambda).\]
From the first relation we deduce $u\in\sm$. Using the second one and $m<\tilde m$ in $\Omega$, we get
\[C(m,\tilde m,\lambda) = \int_\Omega\tilde m(x)|u|^p\,dx > \int_\Omega m(x)|u|^p\,dx = 1,\]
which proves \eqref{mos1} (incidentally, we have also proved that $C(m,\tilde m,\lambda)$ is attained). 
\vskip2pt
\noindent
Now we shall prove our assertion. By Proposition \ref{mok} we have $\lambda_k(m)\ge\lambda_k(\tilde m)$. By Proposition \ref{scd}, for all $n\in\N$ there exists $f_n\in C_2(\s,\sm)$ s.t.\
\beq\label{mos2}
\max_{\omega\in\s}\|f_n(\omega)\|^p \le \lambda_2(m)+\frac{1}{n}.
\eeq
Since $m<\tilde m$ in $\Omega$, for all $\omega\in\s$ we have
\[\int_\Omega\tilde m(x)|f_n(\omega)|^p\,dx > \int_\Omega m(x)|f_n(\omega)|^p\,dx = 1.\]
So we may set for all $n\in\N$, $\omega\in\s$
\[\tilde f_n(\omega) = \frac{f_n(\omega)}{\Big[\int_\Omega\tilde m(x)|f_n(\omega)|^p\,dx\Big]} \in \mathcal{S}_{s,p}(\tilde m).\]
Using Lemma \ref{cmp} (with weight $\tilde m$) we see that $\tilde f_n\in C_2(\s,\mathcal{S}_{s,p}(\tilde m))$. By compactness of $\s$, we can find $\overline\omega\in\s$ s.t.\
\[\|\tilde f_n(\overline\omega)\|^p = \max_{\omega\in\s}\|\tilde f_n(\omega)\|^p.\]
Clearly $\lambda_2(m)+1>\lambda_1(m)$, so we may define $C(m,\tilde m,\lambda_2(m)+1)$ as above. Now, by Proposition \ref{scd} (characterization of $\lambda_2(\tilde m)$), along with \eqref{mos2} and the definition of $C(m,\tilde m,\lambda_2(m)+1)$ we have
\[\lambda_2(\tilde m) \le \|\tilde f_n(\overline\omega)\|^p = \frac{\|f_n(\overline\omega)\|^p}{\int_\Omega\tilde m(x)|f_n(\overline\omega)|^p\,dx} \le \frac{\lambda_2(m)+1/n}{C(m,\tilde m,\lambda_2(m)+1)}.\]
Passing to the limit as $n\to\infty$ and using \eqref{mos1} we get
\[\lambda_2(\tilde m) \le \frac{\lambda_2(m)}{C(m,\tilde m,\lambda_2(m)+1)} < \lambda_2(m),\]
thus concluding the proof.
\end{proof}

\noindent
{\bf Conclusion.} Simply lining up Propositions \ref{exi}, \ref{first}, \ref{scd}, \ref{mok}, \ref{mof}, and \ref{mos}, we have the full proof of Theorem \ref{main}.
\vskip4pt
\noindent
{\bf Acknowledgement.} The author is a member of GNAMPA (Gruppo Nazionale per l'Analisi Matematica, la Probabilit\`a e le loro Applicazioni) of INdAM (Istituto Nazionale di Alta Matematica 'Francesco Severi') and is supported by the research project \emph{Evolutive and Stationary Partial Differential Equations with a Focus on Biomathematics}, funded by Fondazione di Sardegna (2019), and by the grant PRIN-2017AYM8XW: \emph{Nonlinear Differential Problems via Variational, Topological and Set-valued Methods}. The author thanks the anonymous Referee for her/his careful reading of the manuscript and useful suggestions. Finally, the author acknowledges Prof.\ L.\ Brasco's precious help: 'I can no other answer make, but, thanks, and thanks' (W.\ Shakespeare, {\em The twelfth night}, III, 3).

\end{document}